\documentclass[11pt]{amsart}

\usepackage{amsfonts,amsmath,amssymb,amsthm,amscd,amsxtra}
\usepackage{enumerate,epsfig,verbatim}
\usepackage{mathptmx}
\usepackage{amsfonts,amsmath,amssymb,amsthm,amscd,amsxtra}
\usepackage{enumerate,verbatim}
\usepackage{centernot}

\usepackage[notcite,notref]{}

\usepackage{tikz}

\usepackage[all,2cell,ps]{xy}
\usepackage{amsfonts}
\usepackage[margin=1.4in]{geometry}
\usepackage{color}
\usepackage[notcite,notref]{}
\usepackage{url}
\usepackage{datetime}
\usepackage{mathtools}

\usepackage[all,2cell,ps]{xy}
\usepackage{amssymb}
\usepackage{amsmath}
\usepackage{amsfonts}
\usepackage{amsmath}
\usepackage{amsthm}
\usepackage{amssymb}
\usepackage{amscd}
\usepackage{amsfonts}
\usepackage{amsxtra}     % Use various AMS packages
\usepackage{epsfig}
\usepackage{verbatim}
\usepackage[all]{xypic}

\SelectTips{cm}{}

\def\latex/{{\protect\LaTeX}}
\def\latexe/{{\protect\LaTeXe}}
\def\amslatex/{{\protect\AmS-\protect\LaTeX}}
\def\tex/{{\protect\TeX}}
\def\amstex/{{\protect\AmS-\protect\TeX}}
\def\bibtex/{{Bib\protect\TeX}}
\def\makeindx/{\textit{MakeIndex}}

\usepackage{amsmath}
\usepackage{amsthm}
\usepackage{amssymb}
\usepackage{amscd}
\usepackage{amsxtra}     % Use various AMS packages
\usepackage{epsfig}
\usepackage{verbatim}
\usepackage[all]{xypic}
\usepackage{enumerate}

\theoremstyle{plain} %default

\newtheorem{thm}{Theorem}[section]

\newtheorem{prop}[thm]{Proposition}
\newtheorem{cor}[thm]{Corollary}

\theoremstyle{definition}
\newtheorem{chunk}[thm]{\hspace*{-1.065ex}\bf}
\newtheorem{lem}[thm]{Lemma}
\newtheorem{dfn}[thm]{Definition}

\newtheorem{eg}[thm]{Example}

\newtheorem{ques}[thm]{Question}

\newtheorem{rmk}[thm]{Remark}

\newcommand{\fm}{\mathfrak{m}}

\newcommand{\CC}{\mathbb{C}}

\newcommand{\fp}{\mathfrak{p}}
\newcommand{\fq}{\mathfrak{q}}

\DeclareMathOperator{\Tor}{\operatorname{\mathsf{Tor}}}
\DeclareMathOperator{\Ext}{\operatorname{\mathsf{Ext}}}
\DeclareMathOperator{\Hom}{\operatorname{\mathsf{Hom}}}

\DeclareMathOperator{\CI}{\operatorname{\mathsf{CI-dim}}}

\DeclareMathOperator{\Tr}{\textnormal{Tr}}

\DeclareMathOperator{\G-dim}{\operatorname{\mathsf{G-dim}}}

 \DeclareMathOperator{\Ass}{Ass}
 
 \DeclareMathOperator{\Supp}{Supp}
 \DeclareMathOperator{\Spec}{Spec}

 \DeclareMathOperator{\pd}{\operatorname{\mathsf{pd}}}
 
 \DeclareMathOperator{\height}{\operatorname{\mathsf{height}}}

 \DeclareMathOperator{\depth}{\operatorname{\mathsf{depth}}}

  \DeclareMathOperator{\rk}{\operatorname{\mathsf{rank}}}

  \DeclareMathOperator{\Y}{\operatorname{\mathsf{Y}}}

 \def\Tr{\mathsf{Tr}\hspace{0.01in}}

 \newcommand{\Min}{\textup{Min}}

\def\urltilda{\kern -.15em\lower .7ex\hbox{\~{}}\kern
  .04em}\def\urldot{\kern -.10em.\kern -.10em}\def\urlhttp{http\kern
  -.10em\lower -.1ex\hbox{:}\kern -.12em\lower 0ex\hbox{/}\kern
  -.18em\lower 0ex\hbox{/}}

\newcommand{\bb}{\left[ \begin{smallmatrix}}
\newcommand{\eb}{\end{smallmatrix} \right]}

\begin{document}

\title[On an example concerning the second rigidity theorem]{On an example concerning the second rigidity theorem}

%\date{\today}

\author[O. Celikbas]{Olgur Celikbas}
\address{Olgur Celikbas\\
Department of Mathematics \\
West Virginia University\\
Morgantown, WV 26506-6310, U.S.A}
\email{olgur.celikbas@math.wvu.edu}

\author[H. Matsui]{Hiroki Matsui} 
\address{Hiroki Matsui\\Graduate School of Mathematical Sciences\\ University of Tokyo, 3-8-1 Komaba, Meguro-ku, Tokyo 153-8914, Japan}
\email{mhiroki@ms.u-tokyo.ac.jp}

\author[A. Sadeghi]{Arash Sadeghi}
\address{Arash Sadeghi\\School of Mathematics, Institute for Research in Fundamental Sciences (IPM), P.O. Box: 19395-5746, Tehran, Iran.}
\email{sadeghiarash61@gmail.com}

\thanks{Sadeghi's research was supported by a grant from the Institute for Research in Fundamental Sciences (IPM). Matsui's research was supported by JSPS Grant-in-Aid for Young Scientists 19J00158}
\maketitle{}

\begin{abstract} In this paper we revisit an example of Celikbas and Takahashi concerning the reflexivity of tensor products of modules. We study Tor-rigidity and the Hochster--Huneke graph with vertices consisting of minimal prime ideals, and determine a condition with which the aforementioned example cannot occur. Our result, in particular, corroborates the Second Rigidity Theorem of Huneke and Wiegand. 
\end{abstract}

\section{Introduction}

Throughout $R$ denotes a commutative Noetherian local ring with unique maximal ideal $\fm$, and all $R$-modules are assumed to be finitely generated. For unexplained notations and terminology, such as the definitions of homological dimensions, we refer the reader to \cite{AV2, BH, Mas}.

In this paper we are concerned with the following result of Huneke and Wiegand, which is known as the \emph{Second Rigidity Theorem}; see \cite[2.1]{HW1}.

\begin{thm} $($Huneke and Wiegand \cite{HW1}$)$ \label{HWS} Let $R$ be a hypersurface ring, and let $M$ and $N$ be $R$-modules such that $M$ has rank, i.e., there is a nonnegative integer $r$ such that $M_{\fp}$ is free of rank $r$ for each associated prime ideal $\fp$ of $R$ (e.g., $\pd_R(M)<\infty$). If $M\otimes_RN$ is reflexive, or in this context equivalently, is a second syzygy module, then $N$ is reflexive. \pushQED{\qed} 
\qedhere
\popQED	
\end{thm}

Another conclusion of Theorem \ref{HWS}, which is worth noting, is the vanishing of $\Tor_i^R(M,N)$ for each $i\geq 1$. For quite some time it has been an open problem whether the module $M$ in Theorem \ref{HWS} must also be reflexive; see \cite{CHW}. Recently Celikbas and Takahashi \cite{CT2} has given an example disproving this query: there is a reduced hypersurface ring $R$, and modules $M$ and $N$ over $R$ such that both $M\otimes_RN$ and $N$ are reflexive, $\pd_R(M)<\infty$, but $M$ is not reflexive. Moreover, it can be easily checked that there exists a prime ideal $\fq$ of $R$ of height one such that the module $N$ in the example satisfies $\pd_{R_{\fq}}(N_{\fq})=\infty$; see Example \ref{cteg} for details. The main aim of this paper is to show that such an example cannot occur in case $\pd_{R_{\fp}}(N_{\fp})<\infty$ for each prime ideal $\fp$ of $R$ of height at most one. More precisely, we prove:

\begin{thm}\label{anasonuc} Assume $R$ is a hypersurface ring (quotient of an unramified regular local ring), and $M$ and $N$ are nonzero $R$-modules. Assume further:
\begin{enumerate}[\rm(i)] 
\item $\pd(M)<\infty$.
\item $\pd_{R_{\fp}}(N_{\fp})<\infty$ for each prime ideal $\fp$ of $R$ of height at most one.
\end{enumerate}
If $M\otimes_RN$ is reflexive, then both $M$ and $N$ are reflexive.
\end{thm}

We give a proof of Theorem \ref{anasonuc} in section 4, but in fact our main argument is more general: we consider tensor products $M\otimes_RN$ which are $n$-th syzygy modules for $n\geq 2$ and  modules $N$ of finite complete intersection dimension over rings that are not necessarily hypersurfaces; see Theorem \ref{l4}. A key ingredient of our proof is the fact that, when $R$ satisfies Serre's condition $(S_2)$, the \emph{Hochster-Huneke graph} \cite{HH} is connected; see Theorem \ref{rk}.

\section{Preliminaries}

\begin{chunk} \label{c2} An $R$-module $M$ is said to be \emph{Tor-rigid} provided that the following condition holds: if $N$ is an $R$-module with $\Tor_1^R(M,N)=0$, then $\Tor_2^R(M,N)=0$.  Examples of Tor-rigid modules are abundant in the literature. For example, each syzgy of the $R$-module $M$ is Tor-rigid if:
\begin{enumerate}[\rm(i)]
\item $R$ is a hypersurface that is quotient of an unramified regular local ring, and $M$ has either finite length or finite projective dimension; see \cite[2.4]{HW1} and \cite[Theorem 3]{Li}.
\item $R$ has positive depth and $M=\fm^r$ for some integer $r\geq 1$; see \cite[2.5]{CT3}. \pushQED{\qed} 
\qedhere
\popQED	
\end{enumerate}
%On the other hand, if $R$ is Cohen-Macaulay with canonical module $\omega$ and $R$ is not Gorenstein, then $\omega$ is not Tor-rigid \cite{}.
\end{chunk}

\begin{chunk} \label{c0}  Let $M$ be an $R$-module with a projective presentation $P_1\overset{f}{\rightarrow}P_0\rightarrow M\rightarrow 0$. The \emph{transpose} $\Tr M$ of $M$
is the cokernel of $f^{\ast}=\Hom_{R}(f,R)$, and hence is given by the exact sequence:
$0\rightarrow M^*\rightarrow P_0^*\rightarrow P_1^*\rightarrow \Tr M\rightarrow 0$.
Note $\Tr M$ is well-defined up to projective summands.

Given an integer $n\geq 0$, it follows from \cite[2.8]{AuBr} that there is an exact sequence of functors:
%\begin{equation}\notag{}
$$0 \rightarrow \Ext^1_R(\Tr \Omega^n M,-) \rightarrow \Tor_n^R(M,-) \rightarrow \Hom_R(\Ext^n_R(M,R),-)
\rightarrow \Ext^2_R(\Tr \Omega^n M,-). \;\;\;\;\;\;\; \qed $$
%\qedhere
%\popQED	
%\end{equation}
\end{chunk}

Recall that an $R$-module $N$ is said to be \emph{torsionless} if the natural map $N \to N^{\ast\ast}$ is injective, i.e., $\Ext^1_R(\Tr N, R)=0$; see \ref{c0}.

\begin{chunk} \label{c3} Let $N$ be a torsionless $R$-module and let $\{f_{1},f_{2},\dots, f_{s}\}$ be a minimal generating set of the module $N^{\ast}=\Hom(N,R)$. Let $\displaystyle{\delta: R^{\oplus s} \twoheadrightarrow N^{\ast}}$ be defined by $\delta(e_{i})=f_{i}$ for $i=1,2,\dots, s$, where $\{e_{1},e_{2},\dots, e_{s}\}$ is the standard basis for $R^{\oplus s}$. Then, composing the natural injective map $N \hookrightarrow N^{\ast\ast}$ with $\delta^{\ast}$, we obtain the short exact sequence:
\begin{equation}\notag{}
0 \rightarrow N \stackrel{u}{\rightarrow} R^{\oplus s} \rightarrow N_{1} \rightarrow 0,
\end{equation}
where $u(x)=(f_{1}(x),f_{2}(x),\dots, f_{s}(x))$ for all $x\in N$; see \ref{c0}. Any module $N_{1}$ obtained in this way is called a \emph{pushforward} (or \emph{left projective approximation}) of $M$; see \cite{CFS, EG}. Note that such a construction is unique, up to a non-canonical isomorphism; see, for example, \cite[page 62]{EG}.
Also it follows $\Ext^1_R(N_1,R)=0$ so that $\Omega\Tr N \cong \Tr N_1$ (up to free summands); see \cite[3.9]{AuBr}. \pushQED{\qed} 
\qedhere
\popQED	
\end{chunk}

%\begin{chunk} \label{c4} Let $M$ and $N$ be $R$-modules with $M\neq 0$. Then, by \cite[2.8]{AuBr}, we have an exact sequence:
%\begin{equation}\notag{}
%\Tor^R_2(\Tr \Omega \Tr N, M) \rightarrow \Ext^1_R(\Tr N,R)\otimes_RM \rightarrow \Ext^1_R(\Tr N,M) \rightarrow \Tor^R_1(\Tr \Omega \Tr N, M) \rightarrow 0.
%\end{equation}
%Therefore, if $\Tor^R_2(\Tr \Omega \Tr N, M)=0=\Ext^1_R(\Tr N,M)$, then $\Ext^1_R(\Tr N,R)=0$. In particular, if $\Ext^1_R(\Tr N,M)=0$ and $M$ is Tor-rigid, then $\Ext^1_R(\Tr N,R)=0$. \pushQED{\qed} 
%\qedhere
%\popQED	
%\end{chunk}

\begin{chunk} \label{c1} Let $M$ be an $R$-module and let $n\geq 0$ be an integer. Then $M$ is said to satisfy $(\widetilde{S_n})$ provided $\depth_{R_{\fq}}(M_{\fq})\geq \min\{n, \depth(R_{\fq})\}$ for each $\fq \in \Supp(M)$ (note $\depth(0)=\infty$) If $R$ is Cohen-Macaulay, then $M$ satisfies $(\widetilde{S_n})$ if and only if $M$ satisfies  \emph{Serre's condition} $(S_{n})$; see \cite{EG}. 
\pushQED{\qed} 
\qedhere
\popQED	
\end{chunk}

\begin{chunk}
Given an integer $s\geq 0$, we set $\Y^s(R) = \{\fp \in \Spec R \mid \depth(R_{\fp})\leq s\}$. In particular, $\Y^{0}(R)$ denotes the set of all \emph{associated} prime ideals of $R$.
\end{chunk}

\begin{chunk} \label{c1.1} (\cite[2.4]{DSL} and \cite[3.8]{EG})
Let $M$ be an $R$-module and let $n\geq 1$ be an integer. Assume that $\G-dim_{R_{\fp}}(M_{\fp})<\infty$ for each $\fp \in \Y^{n-1}(R)$. Then the following conditions are equivalent:
\begin{enumerate}[\rm(i)]
\item $M$ satisfies $(\widetilde{S_n})$.
%\item Each $R$-regular sequence of length at most $n$ is $M$-regular.
\item $M$ is \emph{$n$-torsion-free}, i.e., $\Ext^i_R(\Tr M, R)=0$ for each $i=1, \ldots, n$.
\item $M$ is an \emph{$n$-th syzygy} module, i.e., $M\cong \Omega^n(N)$ for some $R$-module $N$. \pushQED{\qed} 
\qedhere
\popQED	
\end{enumerate}
\end{chunk}

\begin{chunk} \label{c4.5} Let $M$ and $N$ be $R$-modules with $\CI(M)<\infty$ or $\CI(N)<\infty$. If $\Tor_i^R(M,N)=0$ for each $i\geq 1$, then $\depth(M)+\depth(N)=\depth(R)+\depth(M \otimes_R N)$, i.e., the \emph{depth formula} holds; see \cite[2.5]{ArY}. \pushQED{\qed} 
\qedhere
\popQED	
\end{chunk}

\begin{chunk} \label{c5} Let $M$ and $N$ be $R$-modules such that $\CI(N)=0$. Then $\Ext^i_R(\Tr N, M)=0$ for all $i\geq 1$ if and only if $\Tor_i^R(M,N)=0$ for all $i\geq 1$; see \cite[3.2]{Bounds}. \pushQED{\qed} 
\qedhere
\popQED	
\end{chunk}

\section{Main theorem}

In this section we will prove the following theorem which is our main result:

\begin{thm}\label{l4} Let $M$ and $N$ be nonzero $R$-modules, and let $n\geq 1$ be an integer. Assume:
\begin{enumerate}[\rm(i)] 
\item $M$ is Tor-rigid.
\item $\CI(N)<\infty$.
\item $M\otimes_RN$ satisfies $(\widetilde{S_n})$.
\item $\Tor_i^R(M,N)$ is torsion for all $i\gg 0$.
\end{enumerate}
Then $\Tor_i^R(M,N)=0$ for all $i\geq 1$, and $N$ satisfies $(\widetilde{S_n})$.
\pushQED{\qed} 
\qedhere
\popQED	
\end{thm}

To prove Theorem \ref{l4}, we will establish several lemmas.

\begin{lem} \label{l1} Let $0 \to N  \stackrel{\mu}
     {\longrightarrow} F \to N_1 \to 0$ be a short exact sequence of $R$-modules, where $F$ is free and $\Ext^1_R(N_1,R)=0$. Then it follows that $\Ext^1_R(\Tr N, M) \cong \Tor_1^R(N_1,M)$.
\end{lem}

\begin{proof} We consider the following commutative diagram, where the horizontal maps are the natural ones and $\Hom(\mu^{\ast}, \;M)$ is injective:
\vspace{-0.5in}
\begin{center}
$\xymatrix{
       \\
&   &  M\otimes_RN \ar[d]^{\mu\otimes M} \ar[r]^{\chi\;\;\;\;\;\;\;\;} & \Hom_R(N^{\ast},M) \ar[d]^{\Hom(\mu^{\ast}, \;M)}    \\
&  & M\otimes_RF \ar[r]^{\cong\;\;\;\;\;\;\;\;}      &  \Hom_R(F^{\ast},M)    
 &  &   
&  }$
\end{center}
Note that $\Ext^1_R(\Tr N,M)=\ker(\chi)$; see \ref{c0}. Hence it follows from the above diagram that $\Ext^1_R(\Tr N,M)=\ker(\chi) \cong \ker(\mu\otimes M)=\Tor_1^R(N_1,M)$, as required.
\end{proof}

\begin{lem} \label{l2} Let $M$ and $N$ be $R$-modules with $\CI_R(N)<\infty$. If $\Tor_i^R(M,N)$ is torsion for all $i\gg 0$, then $\Tor_i^R(M,N)$ and $\Ext^i_R(\Tr N, M)$ are torsion for each $i\geq 1$; cf.,  \cite[A.2.]{GORS2}.
\end{lem}

\begin{proof} Let $\fp \in \Y^{0}(R)$. Then, since $\Tor_i^{R_{\fp}}(M_{\fp},N_{\fp})=0$ for all $i\gg 0$ and $\CI_{R_{\fp}}(N_{\fp})=0$, we conclude that $\Tor_i^{R_{\fp}}(M_{\fp},N_{\fp})=0$ for all $i\geq 1$ and also $\Ext^i_{R_{\fp}}(\Tr_{R_{\fp}}N_{\fp}, M_{\fp})=0$ for all $i\geq 1$; see \ref{c5} and \cite[4.9]{AvBu}.
\end{proof}

\begin{lem} \label{l3} Let $M$ and $N$ be $R$-modules such that $M\neq 0$ and $M$ is Tor-rigid. If $n\geq 1$ is an integer and $\Ext^n_R(N, M)=0$, then $\Ext^n_R(N, R)=0$.
\end{lem}

\begin{proof} It follows from \cite[2.8(b)]{AuBr} that there is an exact sequence:
\begin{equation}\notag{}
\Tor^R_2(\Tr \Omega^n N, M) \rightarrow \Ext^n_R(N,R)\otimes_RM \rightarrow \Ext^n_R(N,M) \rightarrow \Tor^R_1(\Tr \Omega^n N, M) \rightarrow 0.
\end{equation}
As $\Ext^n_R(N, M)=0$ and $M$ is Tor-rigid, we have that $\Tor^R_2(\Tr \Omega^n N, M)=0$. Thus we conclude $\Ext^n_R(N,R)\otimes_RM \cong \Ext^n_R(N,M)=0$. This gives, since $M\neq 0$, that $\Ext^n_R(N, R)=0$.
\end{proof}
%Therefore, if $\Tor^R_2(\Tr \Omega^k \Tr N, M)=0=\Ext^k_R(\Tr N,M)$, then $\Ext^1_R(\Tr N,R)=0$. In particular, if $\Ext^1_R(\Tr N,M)=0$ and $M$ is Tor-rigid, then $\Ext^1_R(\Tr N,R)=0$.
%This establishes the lemma for the case where $n=1$.
%Next we assume $n\geq 2$. Recall $\Omega \Tr N \cong \Tr N_1$; see \ref{c1.1} and \ref{c3}. Therefore it follows that $0=\Ext^2_R(\Tr N,M)\cong \Ext^1_R(\Tr N_1,M)$. Hence, by making use of the $n=1$ case, we conclude that $\Ext^1_R(\Tr N_1, R)=0$. This establishes the case where $n=2$ since $\Ext^2_R(\Tr N, R) \cong \Ext^1_R(\Tr N_1, R)$. Proceeding in this way, we can prove the required claim.

We are now ready to give a proof for our main result:

\begin{proof}[Proof of Theorem \ref{l4}] Note, to show $N$ satisfies $(\widetilde{S_n})$, in view of \ref{c1.1} and Lemma \ref{l3}, it suffices to prove $\Ext^i_R(\Tr N, M)=0$ for each $i=1, \ldots, n$. The vanishing of $\Ext^i_R(\Tr N, M)$, as well as that of $\Tor_i^R(M,N)$, is clear if $\depth(R)=0$; see Lemma \ref{l2}. So we assume $\depth(R)\geq 1$.

It follows from \ref{c0} that there is an injection: $\Ext^1_R(\Tr N, M) \hookrightarrow M\otimes_RN$. It is easy to see, since $M\otimes_RN$ satisfies $(\widetilde{S_1})$, that $M\otimes_RN$ is torsion-free.
On the other hand, $\Ext^1_R(\Tr N, M)$ is torsion; see Lemma \ref{l2}. This establishes the theorem for the case where $n=1$, and also yield the vanishing of $\Tor_{i}^R(M,N)$ as we observe next: it follows from Lemma \ref{l3} that $\Ext^1_R(\Tr N, R)=0$, and hence we can consider the pushforward $N_1$ of $N$; see \ref{c1.1} and \ref{c3}. Now Lemma \ref{l1} shows $\Ext^1_R(\Tr N, M)=0=\Tor_1^R(N_1,M)$. As $M$ is Tor-rigid, we have $\Tor_i^R(M,N)=0$ for each $i\geq 1$.

%Suppose for now that $\Ext^1_R(\Tr N, M)\neq 0$, and seek a contradiction. Note there exists $\fp \in \Spec(R)$ such that $\Ext^1_R(\Tr N, M)_{\fp}$ is a nonzero module of depth zero. In particular, as $\Ext^1_R(\Tr N, M)$ is torsion, we conclude that $\depth(R_{\fp})\geq 1$; see Lemma \ref{l2}. This gives us the required conclusion since $(M\otimes_RN)_{\fp}$ is a nonzero module of positive depth, and hence cannot admit a submodule of depth zero. Therefore we see $\Ext^1_R(\Tr N, M)=0$. This establishes the theorem for the case where $n=1$, and also yield the vanishing of $\Tor_{i}^R(M,N)$ as we observe next: it follows from Lemma \ref{l3} that $\Ext^1_R(\Tr N, R)=0$, and hence we can consider the pushforward $N_1$ of $N$; see \ref{c1.1} and \ref{c3}. Now Lemma \ref{l1} shows that $0=\Ext^1_R(\Tr N, M)=\Tor_1^R(N_1,M)$. Also, since $M$ is Tor-rigid, we obtain the vanishing of $\Tor_i^R(M,N)$ for each  $i\geq 1$.
%As $\Ext^1_R(\Tr N, M)$ is torsion, we conclude that $\Ext^1_R(\Tr N, M)=0$; see Lemma \ref{l2}. This establishes the claim for $n=1$ case. Moreover, by \ref{c4}, we have $\Ext^1_R(\Tr N, R)=0$ so that we can consider the pushforward $N_1$ of $N$; see \ref{c1.1} and \ref{c3}. Now Lemma \ref{l1} shows that $0=\Ext^1_R(\Tr N, M)=\Tor_1^R(N_1,M)$. Also, since $M$ is Tor-rigid, we have the vanishing of $\Tor_i^R(M,N)$ for each  $i\geq 1$.Assume $n=1$. and proceed by induction on $n$. 

Next we assume $n\geq 2$, and proceed by induction on $n$ to show that $N$ satisfies $(\widetilde{S_n})$. Suppose there is an integer $t$ such that $t<n$ and $\Ext^i_R(\Tr N, M)=0$ for each $i=1, \ldots, t$. 
Our aim is to prove the vanishing of $\Ext^{t+1}_R(\Tr N, M)$.

It follows from Lemma \ref{l3}, that $\Ext^i_R(\Tr N, R)=0$ for each $i=1, \ldots, t$, i.e., $N$ satisfies $(\widetilde{S_t})$. Therefore we can consider the pushforward sequences 
\begin{equation}\tag{\ref{l4}.1}
0 \to N_{i-1} \to F_i \to N_i \to 0, 
\end{equation}
where $N_0=N$, $F_i$ is free and $\Ext^1_R(N_i,R)=0$ for each $i=1, \ldots, t$; see \ref{c3}.

Note that, for each $i=1,\ldots, t$, we have:
\begin{equation}\tag{\ref{l4}.2}
\Tor_1^R(M,N_i) \cong \Ext^1_R(\Tr N_{i-1}, M) \cong \Ext^i_R(\Tr N, M)=0.
\end{equation}
Here, the first isomorphism in (\ref{l4}.2) is due to Lemma \ref{l1}, while the second isomorphism follows since $\Omega^{i-1}\Tr N \cong \Tr N_{i-1}$ for $i=1, \ldots t$; see \ref{c3}. 

Now, in view of (\ref{l4}.2), tensoring the short exact sequences in (\ref{l4}.1) with $M$, we obtain the following short exact sequences for each $i=1,\ldots, t$:
\begin{equation}\tag{\ref{l4}.3}
0 \to M\otimes_RN_{i-1} \to M\otimes_RF_i \to M\otimes_RN_i \to 0.
\end{equation}

Recall our aim is to show that $\Ext^{t+1}_R(\Tr N, M)=0$, and since $\Omega^t\Tr N \cong \Tr N_t$ (up to free summands), we have $\Ext^{t+1}_R(\Tr N, M)\cong \Ext^1_R(\Tr N_t, M)$; see \ref{c3}. So \ref{c0} yields an injection  as:
\begin{equation}\tag{\ref{l4}.4}
\Ext^{t+1}_R(\Tr N, M) \hookrightarrow M\otimes_RN_{t}.
\end{equation}
Next we assume $\Ext^{t+1}_R(\Tr N, M)\neq 0$, pick $\fq\in \Ass(\Ext^{t+1}_R(\Tr N, M))$, and seek a contradiction.

Suppose $\fq \in \Y^{t}(R)$. Then, since $N$ satisfies $(\widetilde{S_t})$, we have $\depth_{R_{\fq}}(N_{\fq})\geq \depth(R_{\fq})$. This shows $\CI_{R_{\fq}}(N_{\fq})=\depth(R_{\fq})-\depth_{R_{\fq}}(N_{\fq})=0$. Therefore, since $\Tor_i^R(M,N)=0$ for each  $i\geq 1$, we deduce from \ref{c5} that $\Ext^i_R(\Tr N, M)_{\fq}=0$ for each  $i\geq 1$. In particular $\fq \notin \Y^{t}(R)$, i.e., $\depth(R_{\fq})\geq t+1$, because of the fact that $\Ext^{t+1}_R(\Tr N, M)_{\fq}\neq 0$.

Notice $\fq \in \Supp(M)\cap\Supp(N)$. Hence it follows from \ref{c4.5} that 
\begin{equation}\tag{\ref{l4}.5}
\depth_{R_{\fq}}(M_{\fq})=\big(\depth(R_{\fq})-\depth_{R_{\fq}}(N_{\fq}) \big)+\depth_{R_{\fq}}(M_{\fq}\otimes_{R_{\fq}} N_{\fq})\geq t+1.
\end{equation}
The inequality in (\ref{l4}.5) are due to the following facts: $t+1\leq n$ so that $M\otimes_RN$ satisfies $(\widetilde{S}_{t+1})$, $\depth(R_{\fq})\geq t+1$, and $\CI_{R_{\fq}}(N_{\fq})=\depth(R_{\fq})-\depth_{R_{\fq}}(N_{\fq})\geq 0$.

Recall that $\Ext^{t+1}_R(\Tr N, M)_{\fq}$ is a nonzero module of depth zero. Hence, we see, by revisiting (\ref{l4}.4), that $\depth_{R_{\fq}}(M_{\fq}\otimes_{R_{\fq}} (N_{t})_{\fq})=0$. However, by localizing (\ref{l4}.3) at $\fq$ and using depth lemma, along with (\ref{l4}.5), we have
$\depth_{R_{\fq}}(M_{\fq}\otimes_{R_{\fq}} N_{\fq})=t$; this is a contradiction since $M\otimes_RN$ satisfies $(\widetilde{S}_{t+1})$ and so $\depth_{R_{\fq}}(M_{\fq}\otimes_{R_{\fq}} N_{\fq})\geq t+1$. Consequently, $\Ext^{t+1}_R(\Tr N, M)$ must vanish, and this completes the proof of the theorem.
\end{proof}

\section{Proof of Theorem \ref{anasonuc} and further remarks}

\begin{dfn} (\cite{HH}) The {\it Hochster-Huneke graph} $G(R)$ is defined as follows:
\begin{itemize}
\item
The set of \emph{vertices} equals $\Min(R)$, i.e., vertices are the minimal prime ideals of $R$.
\item
There is an \emph{edge} between two vertices $\fp$ and $\fq$ of $G(R)$ $\Longleftrightarrow$ $\height (\fp+\fq) \le 1$.
\end{itemize}
\end{dfn}

\begin{rmk}\label{rem} (\cite{HH}) The following hold for the graph $G(R)$:
\begin{enumerate}[\rm(i)]
\item Given two vertices $\fp_1$ and $\fp_2$ of $G(R)$, there is an edge between $\fp_1$ and $\fp_2$ if and only if $\fp_1 + \fp_2$ is contained in some height-one prime ideal $\fq$. 
\item $G(R)$ is \emph{connected} if and only if given two vertices $\fp$ and $\fp'$ of $G(R)$, there are minimal prime ideals  $\{ \fp_0, \fp_1, \ldots, \fp_r \}$ of $R$, and height-one prime ideals 
$\{ \fq_1, \fq_2, \ldots, \fq_r \}$ of $R$, where $\fp=\fp_0$, $\fp'=\fp_r$ and $\fp_i, \fp_{i+1} \subseteq \fq_{i+1}$ for each $i=0, 1, \ldots, r-1$. \pushQED{\qed} 
\qedhere
\popQED	
\end{enumerate}
\end{rmk}

The first part of the next proposition is proved in \cite[3.6]{HH} for complete local rings. Here, for the convenience of the reader, we go over its proof since we do not assume $R$ is complete.

\begin{prop} \label{rk} Assume $R$ satisfies $(S_2)$, e.g., $R$ is Cohen-Macaulay. Then the following hold:
\begin{enumerate}[\rm(i)]
\item $G(R)$ is connected.
\item If $N$ is an $R$-module such that $N_{\fp}$ is free for each $\fp \in Y^{1}(R)$, then $N$ has rank. %satisfying $(\widetilde{S_1})$ 
\end{enumerate}
\end{prop}

\begin{proof}
(i)  We assume $G(R)$ is not connected, and seek a contradiction.

Notice, since $G(R)$ is disconnected, there is a nontrivial partition of the set of all minimal prime ideals of $R$ as $\Min(R) = \{\fp_1, \cdots, \fp_r\} \sqcup \{\fq_1, \cdots, \fq_s\}$, where $\height (\fp_i + \fq_j) \ge 2$ for each $i$ and $j$.
Letting $I= \bigcap_{i=1}^r \fp_i$ and $J= \bigcap_{j=1}^s \fq_j$, we get two non-nilpotent ideals $I$ and $J$ such that $IJ$ is nilpotent.
Moreover it follows that $\height (I+J) \ge 2$ since
\begin{equation}\notag{}
V(I+J) = V(I) \cap V(J) =  \Bigg[ \bigcup_{i=1}^r V(\fp_i) \Bigg] \bigcap  \Bigg[ \bigcup_{j=1}^s V(\fq_j) \Bigg] = \bigcup_{i,j}V(\fp_i + \fq_j) \text{ and } \height (\fp_i + \fq_j) \ge 2.
\end{equation}
By replacing the ideals $I$ and $J$ with their appropriate powers, we may assume $IJ = 0$.

Since $R$ satisfies $(S_2)$ and $\height(I+J) \ge 2$, there is an $R$-regular sequence $\{u+v, u'+v'\}$ in $I+J$, where $u,u' \in I$ and $v, v' \in J$.
In view of the fact $v'(u+v) - v(u'+v') = v'u-vu' \in IJ=0$, we conclude that there is an element $a \in R$ such that $v = a(u+v)$.
Similarly, we deduce that $u = b(u+v)$ for some $b \in R$. Therefore we have $u+v = (a+b)(u+v)$, and hence $a+b$ is unit in $R$. This implies that either $a$ or $b$ is unit in $R$.
We assume, without loss of generality, that $a$ is unit. Then $u$ is $R$-regular, and the equality $uJ = 0$ shows that $J=0$, which is a contradiction. Consequently, $G(R)$ is not connected.

(ii) Note, as $R$ satisfies $(S_2)$, each associated prime of $R$ is minimal, and $\fp \in Y^{1}(R)$ if and only if $\height(\fp)\leq 1$. Moreover, by part (i), we know $G(R)$ is connected. 

Let $\fp$ and $\fp'$ be two minimal prime ideals of $R$. Then we know  there are minimal prime ideals  $\{ \fp_0, \fp_1, \ldots, \fp_r \}$ of $R$, and height-one prime ideals 
$\{ \fq_1, \fq_2, \ldots, \fq_r \}$ of $R$, where $\fp=\fp_0$, $\fp'=\fp_r$ and $\fp_i, \fp_{i+1} \subseteq \fq_{i+1}$ for each $i=0, 1, \ldots, r-1$.

By assumption, for each $i=0, 1, \ldots, r-1$, we know that the modules $M_{\fp_i}$, $M_{\fp_{i+1}}$ and $M_{\fq_{i+1}}$ are free. Moreover, as $(M_{\fq_{i+1}})_{\fp_iR_{\fq_{i+1}}}\cong M_{\fp_i}$, for each $i=0,1,\ldots, r-1$, we deduce:
\begin{equation}\notag{}
\rk_{R_{\fp_i}}(M_{\fp_i}) = \rk_{R_{\fq_{i+1}}}(M_{\fq_{i+1}}) = \rk_{R_{\fp_{i+1}}}(M_{\fp_{i+1}}).
\end{equation}
This shows that $\rk_{R_{\fp}}(M_{\fp}) = \rk_{R_{\fp'}}(M_{\fp'})$, as required.
\end{proof}
%As $N_\fp$ is free for each prime $\fp \in \Min(R)$, it is enough to check that $N_\fp$ has constant rank on $\Min(R)$.
%Take two minimal prime ideals $\fp$ and $\fp'$.there are minimal prime ideals $\fp_0= \fp, \ldots, \fp_r=\fp'$ of $R$, and also height-one prime ideals $\fq_1, \ldots, \fq_r$ of $R$, where $\fp_i, \fp_{i+1} \subseteq \fq_{i+1}$.

We can strengthen the conclusion of Theorem \ref{l4}, and show that both modules in question satisfy $(\widetilde{S_n})$ in case local freeness hypothesis on $\Y^1(R)$ is included in our assumptions. 
 
\begin{cor} \label{maincor} Assume $R$ satisfies $(S_2)$, $n$ is a positive integer, and $M$ and $N$ are nonzero $R$-modules. Assume further:
\begin{enumerate}[\rm(i)]
\item $M$ is Tor-rigid.
\item $\CI(N)<\infty$.
\item $M\otimes_RN$ satisfies $(\widetilde{S_n})$.
%\item $\Tor_i^R(M,N)$ is torsion for all $i\gg 0$.
\item $\pd_{R_{\fp}}(N_{\fp})<\infty$ for each $\fp \in \Y^{1}(R)$. 
\end{enumerate}
Then $\Tor_i^R(M, N) =0$ for all $i \ge 1$, and both $M$ and $N$ satisfy $(\widetilde{S_n})$.
\end{cor}
%and either of the following condition is satisfied.
%\begin{enumerate}
%\item[\rm(iv)]
%$\Tor_i^R(M,N)_{\fp}=0$ for all $i\gg 0$ and $\fp \in \X^{n-1}(R)$, and $N$ is locally free on $\X^1(R)$.
%\item[\rm(v)]
%$M$ is test and locally free on $\X^{1}(R)$.
%\end{enumerate}

\begin{proof} It follows from Theorem \ref{l4} that $\Tor_i^R(M, N) =0$ for all $i \ge 1$, and $N$ satisfies $(\widetilde{S_n})$. Note that, both $M\otimes_RN$ and $N$ satisfy $(S_1)$. Hence $N_{\fp}$ is free for each $\fp \in \Y^{1}(R)$. In particular, $N$ has rank due to Proposition \ref{rk}(ii). Therefore, since $M\otimes_RN$ is torsion-free, we conclude that $\Supp(N)=\Spec(R)$. Now the depth formula shows $M$ satisfies $(\widetilde{S_n})$; see \ref{c4.5} and \cite[1.3]{GO}.
\end{proof}

%First note that if $\Tor_i(M, N) =0$ for all $i \ge 1$ and either $M$ or $N$ has rank, then the other module satisfies $(\widetilde{S_n})$.
%This  easily follows from the depth formula. 
%If (iv) is satisfied, then $M$ and $N$ fulfill the assumptions in Theorem \ref{l4}.
%Therefore, we conclude that $\Tor_i^R(M, N) =0$ for all $i \ge 1$, and $N$ satisfies $(\widetilde{S_n})$.
%Then by Theorem \ref{rk}(ii), $N$ has rank.
%Thus, $M$ also satisfies $(\widetilde{S_n})$.
%Assume (v) is satisfied.
%As in the proof of Theorem \ref{l4}, we get $\Tor_i^R(M, N) =0$ for all $i \ge 1$.
%Since $M$ is test, one has $\pd_R N < \infty$ and in particular has rank.
%Hence by the above remark, $M$ satisfies $(\widetilde{S_n})$.
%Again using Theorem \ref{rk}(ii), $M$ has rank and thus $N$ satisfies $(\widetilde{S_n})$.

Now we can prove Theorem \ref{anasonuc}, the result advertised in the introduction:

\begin{proof}[Proof of Theorem \ref{anasonuc}] The result is an immediate consequence of Corollary \ref{maincor} since $M$ is Tor-rigid by a result of Lichtenbaum; see \ref{c2}(i).
\end{proof}
%Note that, since $R$ is a hypersurface, $\CI_R(N)<\infty$ and $M\otimes_RN$ satisfies $(\widetilde{S_2})$ \cite[]{EG}.
%Moreover, since $\pd(M)<\infty$, it follows that $\Tor_i^R(M,N)_{\fp}=0$ for each prime ideal $\fp$ and all $i\gg 0$, and also $M$ is Tor-rigid \cite[]{Li}. Therefore, by Theorem \ref{l4}, we see 
%that $\Tor_i^R(M,N)=0$ for all $i\geq 1$. As $M$ has rank, we conclude that $N$ satisfies $(\widetilde{S_2})$, i.e., $N$ is reflexive; see \ref{c4.5} and \cite[]{}. Now, by hypothesis (ii) and Theorem \ref{rk}, we deduce that $N$ has rank. In particular, we can use \cite[]{} once more, and see that $M$ is reflexive.

Next we recall an example given in \cite{CT2} concerning the Second Rigidity Theorem; see Theorem \ref{HWS}. The presentation we provide for $M\otimes_RN$ in Example \ref{cteg} has not been given in \cite{CT2} and appears to be new; here we compute it by using \cite{GS, MF}.

\begin{eg}\label{cteg}(\cite{CT2}) Let $R=\CC[\!|x,y,z,w]\!]/(xy)$, $M=\Tr(R/\fp)$, where $\fp=(y,z,w)\in \Spec(R)$, and let $N=R/(x)$. Then $M$ is not reflexive, but since $\pd(M)<\infty$, we have that $N$ is reflexive by Theorem \ref{HWS}. Moreover, $M\otimes_RN$ is reflexive since it is the second syzygy of the cokernel of the rightmost matrix in the following exact sequence:
\begin{center}
	\begin{tikzpicture}
	\node (a) at (0,0) {$R^{\oplus 4}$};
	\node (b) at (3.5,0) {$R^{\oplus 3}$};
	\node (c) at (7,0) {$R^{\oplus 3}$};
	\node (d) at (10.5,0) {$R^{\oplus 4}$};
	
	\node (b') at (5.25,-1.5) {$M\otimes_R N$};
	\node (b'') at (3.5,-3) {$0$};
	\node (c'') at (7,-3) {$0$};
	
	%\node[draw, text width=3cm] at (14,0) {};
	%
	\draw [->] (b'') to (b');
	\draw [->] (b') to (c);
	\draw [->] (b') to (c'');
	\draw [->] (b) to (b');
	\draw [->] (a) edge node[above]{{\scriptsize$\begin{pmatrix}
		x & 0 & 0 & w\\
		0 & x & 0 & y\\
		0 & 0 & x & z
		\end{pmatrix}$}} (b);
	\draw [->] (b) edge node[above]{{\scriptsize$\begin{pmatrix}
			0 & yz & -y^2\\
			-yz & 0 & yz\\
			y^2 & -yw & 0
			\end{pmatrix}$}} (c);
	\draw [->] (c) edge node[above]{{\scriptsize$\begin{pmatrix}
			x & 0 & 0 \\
			0 & x & 0\\
			0 & 0 & x\\
			w & y & z
			\end{pmatrix}$}} (d);
	\end{tikzpicture}
\end{center}
\pushQED{\qed} 
\qedhere
\popQED
\end{eg}	
Next we point out that the conclusion of Theorem \ref{anasonuc} is sharp: 

\begin{rmk} In Example \ref{cteg}, it follows, as $\pd(M)<\infty$, that $\Tor_i^R(M,N)_{\fp}=0$ for all $i\gg 0$ and for all $\fp \in \Spec(R)$, but $M$ is not reflexive. In other words, the torsion hypothesis (iv) of Theorem \ref{l4} is not enough to obtain the conclusion of Theorem \ref{anasonuc}, in general. 

We can easily see that there is a height-one prime ideal $\fq$ of $R$ in Example \ref{cteg} such that $\pd_{R_{\fq}}(N_{\fq})=\infty$. For that note the minimal free resolution of $N$ is given as:
\begin{equation}\notag{}
\xymatrix{\ldots \ar[r]^y& R \ar[r]^x &  R \ar[r]^y &  R \ar[r]^x & R  \ar[r] & N \ar[r]
&0.}
\end{equation}
Localizing this resolution at the height-one prime ideal $\fq=(x,y)$ of $R$, we obtain the minimal free resolution of $N_{\fq}$ over $R_{\fq}$:
\begin{equation}\notag{}
\xymatrix{\ldots \ar[r]^y& R_{\fq} \ar[r]^x &  R_{\fq} \ar[r]^y &  R_{\fq} \ar[r]^x & R_{\fq}  \ar[r] & N_{\fq} \ar[r]
&0.}
\end{equation}
This clearly shows that $\pd_{R_{\fq}}(N_{\fq})=\infty$. \pushQED{\qed} 
\qedhere
\popQED	
\end{rmk}

An $R$-module $M$ is said to be 2-\emph{Tor-rigid} provided, whenever $\Tor_1^R(M,N)=0=\Tor_2^R(M,N)$ for some $R$-module $N$, we have $\Tor_3^R(M,N)=0$. 
We finish this section by noting that the conclusion of Theorem \ref{l4} may fail if the module $M$ is 2-Tor-rigid instead of Tor-rigid:

\begin{eg} \label{egtr} Let $R=\CC[\!|x,y]\!]/(xy)$, $M=R/(x)$ and $N=R/(x^2)$. Note that each $R$-module is 2-Tor-rigid \cite[1.9]{Mu}. Note also that $M\otimes_RN\cong M$ and hence $M\otimes_RN$ satisfies $(\widetilde{S_v})$ for each $v\geq 0$. Also, since $R$ is reduced, $\Tor_i^R(M,N)$ is torsion for each $i\geq 1$. However it is easy to see that $N$ does not satisfy $(\widetilde{S_1})$, $\Tor_1^R(M,N)\neq 0$, and $M$ is not Tor-rigid; see \cite[page 164]{HW2}.  \pushQED{\qed} 
\qedhere
\popQED	
\end{eg}

%Example \ref{egtr} shows that, for the case where $n=1$, the conclusion of Theorem \ref{l4} can easily fail -- if Tor-rigidity hypothesis is removed -- 
It is worth noting that we do not know an example similar to Example \ref{egtr} when $n\geq 2$. More precisely, we ask (cf. Example \ref{cteg}):

\begin{ques} \label{soru} Let $R$ be a hypersurface ring, and let $M$ and $N$ be nonzero $R$-modules. Assume $\Tor_i^R(M,N)$ is torsion for all $i\gg 0$.
If $M\otimes_RN$ is reflexive, then must $M$ \emph{or} $N$ be reflexive?\pushQED{\qed} 
\qedhere
\popQED	
\end{ques}

Notice, if the ring $R$ in Question \ref{soru} is a domain (e.g., an isolated singularity of dimension at least two), then it follows from \ref{c4.5} that both $M$ and $N$ are reflexive; see \cite[1.3]{GO}.

\section*{Acknowledgments}
Part of this work was completed while the authors were visiting the \emph{Nesin Mathematics Village} in May 2019. The authors thank the Village for providing a lively working enviroment.

The authors are grateful to W. Frank Moore for \cite{MF}, and to Yongwei Yao for showing us the proof of Proposition \ref{rk}(ii).


\begin{thebibliography}{10}

\bibitem{ArY}
Tokuji Araya and Yuji Yoshino, \emph{Remarks on a depth formula, a grade
  inequality and a conjecture of {A}uslander}, Comm. Algebra \textbf{26}
  (1998), no.~11, 3793--3806. \MR{MR1647079 (99h:13010)}

\bibitem{AuBr}
Maurice Auslander and Mark Bridger, \emph{Stable module theory}, Memoirs of the
  American Mathematical Society, No. 94, American Mathematical Society,
  Providence, R.I., 1969.

\bibitem{CFS}
Maurice Auslander and Idun Reiten, \emph{Applications of contravariantly finite
  subcategories}, Adv. Math. \textbf{86} (1991), 111--152.

\bibitem{AV2}
Luchezar~L. Avramov, \emph{Infinite free resolutions}, six lectures on
  commutative algebra ({B}ellaterra, 1996), Progr. Math., vol. 166,
  Birkh\"auser, Basel, 1998, pp.~1--118.

\bibitem{AvBu}
Luchezar~L. Avramov and Ragnar-Olaf Buchweitz, \emph{Support varieties and
  cohomology over complete intersections}, Invent. Math. \textbf{142} (2000),
  no.~2, 285--318. \MR{MR1794064 (2001j:13017)}

\bibitem{BH}
Winfried Bruns and J\"{u}rgen Herzog, \emph{Cohen-{M}acaulay rings}, Cambridge
  Studies in Advanced Mathematics, vol.~39, Cambridge University Press,
  Cambridge, 1993.

\bibitem{GORS2}
Olgur Celikbas, Srikanth~B. Iyengar, Greg Piepmeyer, and Roger Wiegand,
  \emph{Criteria for vanishing of tor over complete intersections}, Pacific J.
  Math. \textbf{276} (2015), no.~1, 93--115.

\bibitem{GO}
Olgur Celikbas and Greg Piepmeyer, \emph{Syzygies and tensor product of
  modules}, Math. Z. \textbf{276} (2014), no.~1-2, 457--468.

\bibitem{Bounds}
Olgur Celikbas, Arash Sadeghi, and Ryo Takahashi, \emph{Bounds on depth of
  tensor products of modules}, Journal of Pure and Applied Algebra \textbf{219}
  (2015), no.~5, 1670--1684.

\bibitem{CT2}
Olgur Celikbas and Ryo Takahashi, \emph{On the second rigidity theorem of
  {H}uneke and {W}iegand}, Proc. Amer. Math. Soc. (2019).

\bibitem{CT3}
\bysame, \emph{Powers of the maximal ideal and vanishing of (co)homology},
  preprint; posted at arXiv:1901.04108v1 (2019).

\bibitem{DSL}
Mohammad~T. Dibaei and Arash Sadeghi, \emph{Linkage of modules and the {S}erre
  conditions}, J. Pure Appl. Algebra \textbf{219} (2015), no.~10, 4458--4478.

\bibitem{EG}
E.~Graham Evans and Phillip Griffith, \emph{Syzygies}, London Mathematical
  Society Lecture Note Series, vol. 106, Cambridge University Press, Cambridge,
  1985.

\bibitem{GS}
Daniel~R. Grayson and Michael~E. Stillman, \emph{Macaulay2, a software system
  for research in algebraic geometry}, Available at
  \url{https://faculty.math.illinois.edu/Macaulay2/}.

\bibitem{HH}
Melvin Hochster and Craig Huneke, \emph{Indecomposable canonical modules and
  connectedness}, Contemporary Math. (1994), no.~159, 197--208.

\bibitem{HW1}
Craig Huneke and Roger Wiegand, \emph{Tensor products of modules and the
  rigidity of {T}or}, Math. Ann. \textbf{299} (1994), no.~3, 449--476.

\bibitem{HW2}
\bysame, \emph{Tensor products of modules, rigidity and local cohomology},
  Math. Scand. \textbf{81} (1997), no.~2, 161--183.

\bibitem{CHW}
\bysame, \emph{Correction to ``{T}ensor products of modules and the rigidity of
  {T}or'', {M}ath. {A}nnalen, 299 (1994), 449--476}, Mathematische Annalen
  \textbf{338} (2007), no.~2, 291--293.

\bibitem{Li}
Stephen Lichtenbaum, \emph{On the vanishing of {T}or in regular local rings},
  Illinois J. Math. \textbf{10} (1966), 220--226.

\bibitem{Mas}
Vladimir Ma{\c{s}}ek, \emph{Gorenstein dimension and torsion of modules over
  commutative {N}oetherian rings}, Comm. Algebra \textbf{28} (2000), no.~12,
  5783--5811, Special issue in honor of Robin Hartshorne.

\bibitem{MF}
W.~Frank Moore, \emph{{M}acaulay2 code to compute the biduality map and the
  pushforward module}, unpublished (2015).

\bibitem{Mu}
M.~Pavaman Murthy, \emph{Modules over regular local rings}, Illinois J. Math.
  \textbf{7} (1963), 558--565.

\end{thebibliography}
\end{document}